\documentclass[12pt]{article}
\usepackage[top=2.5cm,bottom=2.5cm,left=2.9cm,right=2.9cm]{geometry}
\usepackage{amssymb}
\usepackage{amsmath,amsthm}
\usepackage{graphicx}
\usepackage{enumerate}
\usepackage{tikz}
\usepackage{mathrsfs}
\usepackage{verbatim}

\usepackage{colortbl}

\newcommand{\n}{\operatorname{n}}
\newcommand{\epn}{\operatorname{epn}}

\newtheorem{remark}{Remark}[section]

\newtheorem{lemma}[remark]{Lemma}
\newtheorem{theorem}[remark]{Theorem}
\newtheorem{proposition}[remark]{Proposition}

\newtheorem{corollary}[remark]{Corollary}

\title{From the strong differential to Italian domination in graphs}

\author{A. Cabrera Mart\'inez, J. A. Rodr\'{\i}guez-Vel\'{a}zquez\\
{\small Universitat Rovira i Virgili }\\{\small Departament d'Enginyeria Inform\`atica i Matem\`atiques } \\  {\small Av. Pa\"{\i}sos
Catalans 26, 43007 Tarragona, Spain.} \\{\small
  abel.cabrera\@@urv.cat, juanalberto.rodriguez\@@urv.cat}}

\date{ }
\begin{document}
\maketitle

\begin{abstract}
Given a  graph $G$ and a subset of vertices $D\subseteq V(G)$, the external neighbourhood of $D$  is defined as $N_e(D)=\{u\in V(G)\setminus D:\, N(u)\cap D\ne \varnothing\}$, where $N(u)$ denotes the open neighbourhood of  $u$. Now, given a subset $D\subseteq V(G)$ and a vertex $v\in D$, the external private neighbourhood of $v$ with respect to $D$ is defined to be $\epn(v,D)=\{u\in V(G)\setminus D: \, N(u)\cap D=\{v\}\}.$  
The strong differential of a set $D\subseteq V(G)$ is defined as $\partial_s(D)=|N_e(D)|-|D_w|,$ where $D_w=\{v\in D:\, \epn(v,D)\ne \varnothing\}$.
In this paper we focus on the study of the strong differential of a graph, which 
is defined as 
$$\partial_s(G)=\max \{\partial_s(D):\, D\subseteq V(G)\}.$$
Among other results, we obtain general bounds on $\partial_s(G)$ and we prove a Gallai-type theorem, which states  that $\partial_s(G)+\gamma_{_I}(G)=\n(G)$, where  $\gamma_{_I}(G)$ denotes the Italian domination number of $G$. Therefore, we can see the theory of strong differential in graphs as a new approach to the theory of Italian domination. One of the advantages of this approach is that it allows us to study the Italian domination number without the use of functions. As we can expect, we derive new results on the Italian domination number of a graph. 
\end{abstract}


\section{Introduction}

Given a graph $G=(V(G), E(G))$ and a vertex $v\in V(G)$, the {open neighbourhood} of $v$ is defined to be $N(v)=\{u\in V(G) :\, uv\in E(G)\}$. 
The {open neighbourhood of a set}  $S\subseteq V(G)$ is defined as 
$N(S)=\cup_{v\in S}N(v)$, while the {external neighbourhood} of $S$, or boundary of $S$,  is defined as $N_e(S)=N(S)\setminus S$. The {differential of a set} $S\subseteq V(G)$ is defined as $\partial(S)=|N_e(S)|-|S|$, while the {differential of a graph} $G$ is defined to be
$$\partial(G)=\max \{\partial(S):\, S\subseteq V(G)\}.$$
As described in \cite{MR2212796}, the definition of $\partial(G)$ was given by Hedetniemi about
twenty-five years ago in an unpublished paper, and was also considered by Goddard
and Henning \cite{MR1605086}. After that, the differential of a graph has been studied by several authors, including \cite{MR3711946,Bermudo-Roman,Sergio-differential-2014,PerfectDifferential,MR2212796,MR2724896}.

%----
Lewis et al.\ \cite{MR2212796} motivated the definition of differential from the
following game, what we call {graph differential game}. ``{You are allowed to buy as many tokens as you like, say $k$ tokens, at a cost of one dollar each. You then place the tokens on some
subset $D$ of $k$ vertices of a graph $G$. For each vertex of $G$ which has no token on
it, but is adjacent to a vertex with a token on it, you receive 
one dollar. Your
objective is to maximize your profit, that is, the total value received
minus the cost of the tokens bought}". Obviously, $\partial(D)=|N_e(D)|-|D|$ is the profit obtained with the placement  $D$, while the maximum profit equals $\partial(G)$.

%-----

Given a set $D\subseteq V(G)$ and a vertex $v\in D$, the {external private neighbourhood} of $v$ with respect to $D$ is defined to be $$\epn(v,D)=\{u\in V(G)\setminus D: \, N(u)\cap D=\{v\}\}.$$ We define the sets $D_w=\{v\in D:\, \epn(v,D)\ne \varnothing\}$ and $D_s=D\setminus D_w$. 

We consider a version of the graph differential game in which you will get a refund of one dollar  for each token placed in a vertex with no external private neighbour with respect to $D$, i.e., 
we will get a refund of one dollar  for a token placed in a vertex $v$ if either every neighbour of $v$ has a token on it or every neighbour of $v$ which has no token on
it  is also adjacent to a vertex different from $v$ with a token on it.
This version of the game can be called {graph differential game with refund}.
 Thus, we define the {strong differential of} $D$, denoted by $\partial_s(D)$, as  the profit obtained with the placement  $D$ in the graph differential game with refund,  which is  $\partial_s(D)=|N_e(D)|-|D_w|$. 
Notice that  
$$\partial_s(D)=|N_e(D)|-|D_w|=|N_e(D)|-|D|+|D_s|=\partial(D)+|D_s|.$$ 
 In the graph differential game with refund,  the maximum profit equals the {strong differential of} $G$, which is defined as
$$\partial_s(G)=\max \{\partial_s(D):\, D\subseteq V(G)\}.$$

For instance, consider the graph $G$ shown in Figure \ref{Fig-Example-2}. If  $D$ is the set of (gray and black) coloured vertices, and we place a token in each vertex belonging to $D$, then in the graph differential game with refund we get a refund of four dollars, as the black-coloured vertices do not have external private neighbours with respect to $D$. Although the gray-coloured vertex has four neighbours not belonging to $D$, it only contributes in three dollars to the profit, as it does not produce any  refund. 
Hence,  
the total profit obtained with the placement $D$ equals  $\partial_s(D)=|N_e(D)|-|D_w|=9-1=8$. Notice that for this graph, the maximum profit equals $\partial_s(G)=\partial_s(D)=8$. 

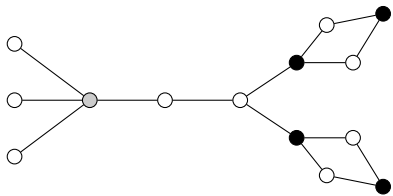
\begin{figure}[h]
\centering
\begin{tikzpicture}[scale=.5, transform shape]

\node [draw, shape=circle, fill=gray!40] (a1) at  (0,0) {};
\node [draw, shape=circle] (a11) at  (-2,0) {};
\node [draw, shape=circle] (a12) at  (-2,1.5) {};
\node [draw, shape=circle] (a13) at  (-2,-1.5) {};

\node [draw, shape=circle] (a2) at  (2,0) {};
\node [draw, shape=circle] (a3) at  (4,0) {};

\node [draw, shape=circle, fill=black] (b1) at  (5.5,1) {};
\node [draw, shape=circle] (b2) at  (7,1) {};
\node [draw, shape=circle, fill=black] (b3) at  (7.8,2.3) {};
\node [draw, shape=circle] (b4) at  (6.3,2) {};

\node [draw, shape=circle, fill=black] (c1) at  (5.5,-1) {};
\node [draw, shape=circle] (c2) at  (7,-1) {};
\node [draw, shape=circle, fill=black] (c3) at  (7.8,-2.3) {};
\node [draw, shape=circle] (c4) at  (6.3,-2) {};

\draw(a11)--(a1)--(a2)--(a3)--(b1)--(b2)--(b3)--(b4)--(b1);
\draw(a3)--(c1)--(c2)--(c3)--(c4)--(c1);
\draw(a1)--(a12);
\draw(a1)--(a13);
\end{tikzpicture}
\caption{A graph $G$ with $\partial_s(G)=8$.}
\label{Fig-Example-2} 
\end{figure}
Suppose that one ``entity" is stationed at some of the vertices of a  graph $G$ and that
an entity at a vertex can deal with a problem at any vertex in its neighbourhood. In general, an
entity could consist of a robot, an observer, a legion, a guard, and so on. In this sense, given a set $D\subseteq V(G)$, we say that all vertices in $N(D)$ are {defended} by the vertices in $D$.
The whole graph is deemed protected under $D$ if every vertex not in $D$ has a neighbour in $D$, and in such a case we say that $D$ is a {protector} of $G$. 
%Notice that a protector is simply a dominating set. 
 A vertex $u\in N_e(D)$ is {strongly defended} under $D$ if $|N(u)\cap D|\ge 2$, otherwise $u$ is {weakly defended} under $D$.
Now, a vertex $v\in D$ is a {weak defender} with respect to $D$
if $\epn(v,D)\ne \varnothing$, as there are vertices which are only defended by  $v$, and so $v$ does not have any help to defend its private neighbours. 
The set of weak defenders with respect to $D$ is $D_w$, while any $v\in D_s=D\setminus D_w$ will be a {strong defender} with respect to $D$.
 In this sense, we can see the strong differential  $\partial_s(D)=|N_e(D)|-|D_w|$   as a manner of quantify the quality of the protection of $G$ with  the placement $D$ of entities. In this case, we
have a penalty  of one unit per each  weak defender in $D$. With this approach, the strong differential allows us to compare two protectors, i.e.,  the larger $\partial_s(D)$ is, the better protector the set $D$ is. 
In Section  \ref{SectionGallai} we will show that for any graph, there exists a protector (dominating set) $D$ with $\partial_s(D)=\partial_s(G)$. Thus, $\partial_s(G)$ quantifies the maximum quality among the protectors  of $G$.
 In the same section  we will see that the concept of strong differential is closely related to the theory of Italian domination, which is one of the approaches to protection of graphs.
In that section, the role of weak and strong defenders will be clearly shown.

The remainder of the paper is organized as follows. 
In  Section \ref{section-tools} we introduce some notation and tools needed to develop  the remaining sections. 
In Section \ref{SectionGallai}  we prove a Gallai-type theorem 
which states  that $\partial_s(G)+\gamma_{_I}(G)=\n(G)$, where  $\gamma_{_I}(G)$ denotes the Italian domination number of $G$. We conclude the section showing  that the problem of finding $\partial_s(G)$  is NP-hard. 
Section  \ref{SectionGeneral Bonds}
 is devoted to provide   general results on the strong differential. We obtain  tight bounds, we show some classes of graphs for which the bounds are achieved  and, in some cases, we characterize the graphs achieving the bounds. Some of these results are derived from known results on the Italian domination number, while from others we can infer new results on this invariant. The paper ends with a brief concluding remark section (Section  
 \ref{SectionConsequences-Gallai}) where we summarize some of these results.

\section{Notation, terminology and basic tools}\label{section-tools}

Throughout the paper, we will use the notation  $G \cong H$ if $G$ and $H$ are isomorphic graphs.  The {closed neighbourhood} of a vertex $v$ is defined as $N[v]=N(v) \cup \{v\}$.  Given a set $S\subseteq V(G)$,  $N[S]=N(S)\cup S$ and the subgraph of $G$ induced by $S$ will be denoted by $G[S]$.  We denote by $\deg(v)=|N(v)|$ the {degree} of vertex $v$, as well as $\delta(G)=\min_{v \in V(G)}\{\deg(v)\}$ the {minimum degree} of $G$,   $\Delta(G)=\max_{v \in V(G)}\{\deg(v)\}$ the {maximum degree} of $G$ and $\n(G)=|V(G)|$ the order of $G$. 
A {leaf} of $G$ is a vertex of degree one. A {support vertex} of $G$ is a vertex which is adjacent to a leaf.  The set of leaves and support~vertices of $G$ will be denoted by $\mathcal{L}(G)$ and $\mathcal{S}(G)$, respectively.

A set $S\subseteq V(G)$ of vertices is a {dominating set}  if
$N(v)\cap S\ne \varnothing$ for every vertex   $v\in V(G)\setminus S$. 
Let $\mathcal{D}(G)$ be the set of dominating sets of $G$. 
The {domination number} of $G$ is defined to be, 
$$\gamma(G)=\min\{|S|:\, S\in \mathcal{D}(G)\}.$$
The domination number  has been extensively studied. For instance, we cite the following books,  \cite{Haynes1998a,Haynes1998}.

 We define a $\gamma(G)$-set as a set $S\in \mathcal{D}(G)$
with $|S|=\gamma(G)$.  The same agreement will be assumed for optimal parameters associated to other characteristic sets defined in the paper. 
 For instance,  a $\partial_s(G)$-set will be  a set $D\subseteq V(G)$ such that $\partial_s(D)=\partial_s(G)$.

A set $S$ of vertices of $G$ is a {vertex cover} if every edge of $G$ is incident with at least one vertex in $S$. The {vertex cover number} of $G$, denoted by $\beta(G)$, is the minimum  cardinality among all vertex covers of $G$. Recall that the largest cardinality of a set of vertices of $G$, no two of which are adjacent, is called the {independence number} of $G$ and it is denoted by $\alpha(G)$.
The following well-known result, due to Gallai, states the relationship between the independence number and the vertex cover number of a graph.

\begin{theorem}[Gallai's theorem, \cite{Gallai1959}]\label{th_gallai}
For any graph  $G$,
$$\alpha(G)+\beta(G) = \n(G).$$
\end{theorem}

Let $f: V(G)\longrightarrow \{0,1,2\}$ be a function and $V_i=\{v\in V(G):\; f(v)=i\}$ for $i\in \{0,1,2\}$. We will identify $f$ with these subsets of $V(G)$ induced by $f$, and write $f(V_0,V_1, V_2).$ The {weight} of $f$ is defined to be $$\omega(f)=f(V(G))=\sum_{v\in V(G)}f(v)=\sum_ii|V_i|.$$

 Cockayne, Hedetniemi and Hedetniemi \cite{Cockayne2004} defined a {\it Roman dominating  function}, abbreviated RDF, on a graph $G$ to be a function $f(V_0,V_1, V_2)$ satisfying the condition that every vertex $u\in V_0$  is adjacent to at least one vertex $v\in V_2$. 
  The \emph{Roman domination number}, denoted by   
    $\gamma_{_R}(G)$, is the minimum weight among all RDFs on $G$, \emph{i.e.}, $$\gamma_{_R}(G)=\min\{\omega(f):\, f \text{ is a RDF on } G\}.$$

A generalization of Roman domination, called Italian domination, was introduced by Chellali et al.\ in \cite{CHELLALI201622}, where it was called Roman $\{2\}$-domination. The concept was studied further in \cite{HENNING2017557,Klostermeyer201920}. An {Italian dominating function}, abbreviated IDF, on a graph $G$   is a
function $f(V_0,V_1,V_2)$ satisfying that $f(N(v))=\sum_{u\in N(v)}f(u)\ge 2$ for every $v\in V_0$, i.e., $f(V_0,V_1,V_2)$ is an IDF if $N(v)\cap V_2\ne \varnothing$ or $|N(v)\cap V_1|\ge 2$ for every $v\in V_0$.

The \emph{Italian domination number}, denoted by  $\gamma_{_I}(G)$, is the minimum weight among all  IDFs on $G$, \emph{i.e.}, $$\gamma_{_I}(G)=\min\{\omega(f):\, f \text{ is an IDF on } G\}.$$

Given a function $f$ on $G$ and a vertex $v\in V(G)$, we can assume that $f(v)$ is the number of entities placed in $v$. In the theory of Italian domination, a graph $G$ is deemed  protected under $f$ if every vertex of weight zero is protected by at least two entities, i.e., if $f$ is an IDF. 
An IDF of weight  $\omega(f)=\gamma_{_I}(G)$ is called a $\gamma_{_I}(G)$-function. In Section \ref{SectionGallai} we will see that for any $\gamma_{_I}(G)$-function $f(V_0,V_1,V_2)$ on $G$, the set $D=V_1\cup V_2$ is a  
 protector of $G$ of maximum quality, as $\partial_s(D)=\partial_s(G)$. Furthermore, $D_w=V_2$ and $D_s=V_1$. Therefore, in the theory of Italian domination, any $\gamma_{_I}(G)$-function provides the minimum number of entities needed to protect the graph and also provides a protector of maximum quality, according to the approach described in the previous section.

We assume that the reader is familiar with the basic concepts, notation  and terminology of domination in graph. If this is not the case, we suggest the textbooks \cite{Haynes1998a,Haynes1998}.  For the remainder of the paper, definitions will be introduced whenever a concept is needed. In particular, this is the case for concepts, notation and terminology to be used only once.

\section{A Gallai-type theorem}\label{SectionGallai}

In order to deduce our results, we need to state the following basic lemma.

\begin{lemma}\label{Strong-dif-dominante}
For any graph $G$, there exists a $\partial_s(G)$-set  which is a dominating set of $G$. 
\end{lemma}

\begin{proof}
Let $D$ be a $\partial_s(G)$-set   such that $|D|$ is maximum among all $\partial_s(G)$-sets. If $D \in \mathcal{D}(G)$, then we are done.  Suppose that $D\not \in \mathcal{D}(G)$. Let  $v\in V(G)\setminus D$ such that $N(v)\cap D=\varnothing$ and let $D'=D\cup \{v\}$. We differentiate two cases.

\vspace{0.2cm}
\noindent
Case 1.  $epn(v,D')\ne \varnothing$. In this case,  $D'_w\subseteq D_w\cup \{v\}$ and $|N(v)\setminus N[D]|\ge 1$, which implies that 
$\partial_s(D')\ge |N_e(D)|+|N(v)\setminus N[D]|-(|D_w|+1)\ge |N_e(D)|-|D_w|=\partial_s(D)=\partial_s(G).$ Therefore, $D'$ is a $\partial_s(G)$-set with $|D'|>|D|$, which is a contradiction.

\vspace{0.2cm}
\noindent
Case 2.  $epn(v,D')= \varnothing$. In this case, $D'_w\subseteq D_w$ and $N_e(D')=N_e(D)$, which implies that $\partial_s(D')\ge |N_e(D)|-|D_w|=\partial_s(D)=\partial_s(G).$ As above, $D'$ is a $\partial_s(G)$-set with $|D'|>|D|$, which is a contradiction.

According to the two cases above, $D$ is a dominating set.  Therefore, the result follows. 
\end{proof}

The following straightforward  remark 
will be one of our tools.
\begin{remark}
\label{eq-2}
If $D\in \mathcal{D}(G)$ is a $\partial_s(G)$-set, then
$$
\partial_s(G)=\n(G)-|D|-|D_w|=\n(G)-|D_s|-2|D_w|.
$$
\end{remark}

%%%%%%%%%%%%%%%%%%%%
%%%%%%%%%%%%%%%

 A Gallai-type theorem has the form $a(G)+ b(G)=\n(G)$, where $a(G)$ and $b(G)$ are parameters defined on $G$.  This terminology comes from Theorem \ref{th_gallai}, which was stated in 1959 by the prolific  Hungarian mathematician
 Tibor Gallai.

% The following Gallai-type theorem was obtained by Bermudo, Fernau and Sigarreta in \cite{Bermudo-Roman}. 

\begin{theorem}[Gallai-type theorem for the differential and the Roman domination number,  \cite{Bermudo-Roman}]\label{Th-Gallai-Roman}
For any graph $G$,
$$\gamma_{_R}(G)+\partial(G)=\n(G).$$
\end{theorem}

Inspired by the previous result, it is natural to ask if we are able to establish a Gallai-type theorem involving the strong differential. Precisely, the next result states the relationship between the strong differential and the Italian domination number.

\begin{theorem}[Gallai-type theorem for the strong differential and the Italian domination number]\label{Th-Gallai-Roman-Strong}
For any graph $G$,
$$\gamma_{_I}(G)+\partial_s(G)=\n(G).$$
\end{theorem}

\begin{proof}
By Lemma \ref{Strong-dif-dominante}, there exists a $\partial_s(G)$-set $D$ which is a dominating set of $G$. Hence, the function $g(W_0,W_1,W_2)$, defined from $W_1=D_s$ and $W_2=D_w$, is an IDF on $G$, which implies that $\gamma_{_I}(G)\le \omega(g)= 2|D_w|+|D_s|$. Therefore, Remark \ref{eq-2} leads to $\gamma_{_I}(G)\le \n(G)-\partial_s(G).$

We proceed to show that $\gamma_{_I}(G)\ge \n(G)-\partial_s(G).$
Let $f(V_0,V_1,V_2)$ be a $\gamma_{_I}(G)$-function.
% where $|V_2|$ is minimum among all $\gamma_{_I}(G)$-functions. 
It is readily seen that for $D'=V_1\cup V_2$ we have  that $D'_s=V_1$ and $D'_w=V_2$. Thus, 
\begin{align*}
\partial_s(G)&\ge \partial_s(D')\\
             &=|N_e(D')|-|D'_w|\\
             &=|V(G)\setminus (V_1\cup V_2)|-|V_2|\\
             &=\n(G)-2|V_2|-|V_1|\\
             &=\n(G)-\gamma_{_I}(G).
\end{align*} 
Therefore, the result follows.
\end{proof}

Theorem \ref{Th-Gallai-Roman-Strong} allows us to derive results on the Italian  domination number from  results on the strong differential  and vice versa. In the next sections we present some of these results.

Now we analyse  the case of the computational complexity.  
Given a positive integer $k$ and a graph $G$, the  problem of deciding if  $G$ has an Italian dominating  function   $f$ of weight  $\omega(f)\le k$  is NP-complete \cite{CHELLALI201622}. Hence, the  problem of computing the Italian domination number of a graph is NP-hard.   Therefore, by Theorem \ref{Th-Gallai-Roman-Strong} we immediately obtain the analogous result  for the strong differential.  

\begin{theorem}
The  problem of computing the strong differential of a graph  is  $NP$-hard.
\end{theorem}

\section{General bounds}\label{SectionGeneral Bonds}

We next present tight bounds on the strong differential of a graph.
In some cases we provide classes of graphs achieving the bounds, while in other cases we characterize the graphs reaching the equalities.

\subsection{Bounds in terms of the order}

In this subsection,  we show some interesting results on the strong differential in terms of the order of $G$. These results are directly derived   from known results on the Italian domination number.

\begin{theorem}{\rm \cite{Klostermeyer201920}}\label{PropKlostermeyer}
If $G$ is a connected graph with $\n(G)\ge  3$, then  $$\gamma_{_I}(G)\leq \frac{3}{4}\n(G).$$
Furthermore, if  $G$ is a connected graph with $\delta(G)\ge  2$, then  $$\gamma_{_I}(G)\leq \frac{2}{3}\n(G),$$
while if $\delta(G)\ge  3$, then  $$\gamma_{_I}(G)\leq \frac{1}{2}\n(G).$$
\end{theorem}

Therefore, from Theorems  \ref{Th-Gallai-Roman-Strong} and \ref{PropKlostermeyer} we derive the following result on the strong differential. 

\begin{theorem}
If $G$ is a connected graph with $\n(G)\ge  3$, then  $$\partial_s(G)\geq \frac{1}{4}\n(G).$$
Furthermore, if  $G$ is a connected graph with $\delta(G)\ge  2$, then  $$\partial_s(G)\geq \frac{1}{3}\n(G),$$
while if $\delta(G)\ge  3$, then  $$\partial_s(G)\geq \frac{1}{2}\n(G).$$
\end{theorem}

The reader is referred to \cite{Haynes2020GraphsWL} for a characterization of all graphs achieving the equalities in the bounds above. 

%-----------------------New----------------------------
\subsection{Bounds in terms of order, domination number,  $2$-domination number and number of support vertices}

The following results show some basic property of the Italian domination number. Recall that a nontrivial tree is a
tree of order at least $2$.

\begin{theorem}\label{LowerBoundTrees} 
\mbox{ }

\begin{itemize}
\item {\rm \cite{Klostermeyer201920}}  If $T$ is a nontrivial tree, then $\gamma_{_I}(T) \ge  \gamma(T)+1.$

\item {\rm \cite{CHELLALI201622}} For every graph $G$, $\gamma_{_I}(G) \le  \gamma_{_R}(G) \le  2\gamma(G).$

\item {\rm \cite{CHELLALI201622}} For every graph $G$, $\gamma_{_I}(G) \le  \gamma_{_2}(G).$
\end{itemize}

\end{theorem}
The trees satisfying $\gamma_{_I}(T)= \gamma(T)+1$, and the trees satisfying $\gamma_{_I}(T)= 2\gamma(T)$, were characterized in \cite{HENNING2017557}.
 
From Theorems \ref{Th-Gallai-Roman-Strong} and \ref{LowerBoundTrees} we conclude that if $T$ is a nontrivial tree, then $\partial_s(T) \le n(T)- \gamma(T)-1.$ As the next result shows, we can state a more general bound, which improves the previous one for some classes of trees.  
Here $\sigma(G)$ denotes the number of support vertices of $G$ which are adjacent to more than one leaf.

\begin{theorem}\label{Prop-RElat-Differentials}
For any graph $G$, 
$$ \n(G)-\min\{2\gamma(G), \gamma_{_2}(G)  \} \le \partial_s(G)\le \n(G)-\gamma(G)-\sigma(G).$$ 
\end{theorem}
\begin{proof}
By Lemma \ref{Strong-dif-dominante}, there exists  $D\in \mathcal{D}(G)$ which is  a $\partial_s(G)$-set. 
For every $x\in \mathcal{S}(G)$ we define $\mathcal{L}(x)=N(x)\cap \mathcal{L}(G)$ and $\mathcal{L}[x]=\{x\}\cup \mathcal{L}(x)$.  Let $\mathcal{S}'(G)=\{x\in \mathcal{S}(G): \, |\mathcal{L}(x)|\ge 2\}$ and $D'=\mathcal{S}'(G)\cup D\setminus(\cup_{x\in \mathcal{S}'(G)}\mathcal{L}(x))$. 
 Notice that $D'$ is a dominating set of $G$ and for any  $x\in \mathcal{S}'(G)$, either $|D\cap \mathcal{L}[x]|=1$ and $x\in D_w$ or $|D\cap \mathcal{L}[x]|\ge 2$, while $|D'\cap \mathcal{L}[x]|=1$ and $x\in D'_w$.  Hence, $\gamma(G)\le |D'|\le |D|$, $\sigma(G)\le |D'_w|$ and $|D'|+|D'_w|\le |D|+|D_w|$. 
 Therefore, by Remark \ref{eq-2}, $$\partial_s(G)= \n(G)-|D|-|D_w|\le\n(G)-|D'|-|D'_w|\le \n(G)-\gamma(G)-\sigma(G).$$

Finally, the lower bounds are derived from Theorems \ref{Th-Gallai-Roman-Strong} and \ref{LowerBoundTrees}.
\end{proof}

We next present some classes of graphs for which the bounds above are achieved. To begin with, we consider the following result, which  is an immediate consequence of Theorem~\ref{Prop-RElat-Differentials}.
 
 \begin{corollary}\label{Corollary=gamma-gamma2}
 Let $G$ be a graph. If $\gamma_{_2}(G)=\gamma(G)$, then $\partial_s(G)=\n(G)-\gamma_{_2}(G)$.
 \end{corollary}

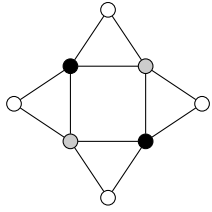
\begin{figure}[ht]
\centering
\begin{tikzpicture}[scale=.5, transform shape]

\node [draw, shape=circle, fill=gray!40] (a1) at  (0,0) {};
\node [draw, shape=circle, fill=black] (a2) at  (2,0) {};
\node [draw, shape=circle, fill=gray!40] (a3) at  (2,2) {};
\node [draw, shape=circle, fill=black] (a4) at  (0,2) {};

\node [draw, shape=circle] (a12) at  (1,-1.5) {};
\node [draw, shape=circle] (a23) at  (3.5,1) {};
\node [draw, shape=circle] (a34) at  (1,3.5) {};
\node [draw, shape=circle] (a41) at  (-1.5,1) {};

%\node [draw, shape=circle] (a1122) at  (1,-3) {};
%\node [draw, shape=circle] (a2233) at  (5,1) {};
%\node [draw, shape=circle] (a3344) at  (1,5) {};
%\node [draw, shape=circle] (a4411) at  (-3,1) {};

\draw(a1)--(a2)--(a3)--(a4)--(a1);
\draw(a1)--(a12)--(a2)--(a23)--(a3)--(a34)--(a4)--(a41)--(a1);
%\draw(a1)--(a1122)--(a2)--(a2233)--(a3)--(a3344)--(a4)--(a4411)--(a1);

\end{tikzpicture}
\caption{The set of (gray and black) coloured vertices forms a $\partial_s(G)$-set and a $\gamma_{_2}(G)$-set, while  the set of black-coloured vertices forms a $\gamma(G)$-set.}
\label{Fig-Example-1} 
\end{figure}

The converse of Corollary \ref{Corollary=gamma-gamma2} does not hold. For instance, for the graph $G$ shown in  Figure~\ref{Fig-Example-1} we have that $\partial_s(G)=4=\n(G)-\gamma_{_2}(G)$, while $\gamma_{_2}(G)=4>2=\gamma(G)$.

 Now, if  $\mathcal{G}$ is the family of graphs of order at least three, where $\{\mathcal{L}(G),\mathcal{S}(G)\}$ is a partition of $V(G)$ and every support vertex is  adjacent to at least two leaves, then for every $G\in \mathcal{G}$ we have that $\gamma(G)=|\mathcal{S}(G)|=\sigma(G)$. In such a case,  the only  $\gamma(G)$-set is $\mathcal{S}(G)$, and we have that $\partial_s(G)\ge \partial_s(S)= |\mathcal{L}(G)|-|\mathcal{S}(G)|=\n(G)-2\gamma(G)$. Therefore, Theorem \ref{Prop-RElat-Differentials}  leads  to the following remark. 

\begin{remark} If $G\in \mathcal{G}$, then  $\partial_s(G)=\partial(G)=\n(G)-2\gamma(G)=\n(G)-\gamma(G)-\sigma(G)$.
\end{remark}

Next we proceed to discuss the case of graphs with $\partial_s(G)=\n(G)-\gamma(G)$.

 \begin{proposition}\label{CharactEqGamma}
 Given a graph $G$, the following statements are equivalent.
 \begin{enumerate}[{\rm (i)}]
 \item $\partial_s(G)=\n(G)-\gamma(G)$.
\item  $\gamma_{_2}(G)=\gamma(G)$.
 \end{enumerate} 
 \end{proposition}
 
\begin{proof}
First, assume $\partial_s(G)=\n(G)-\gamma(G)$.  
By Lemma \ref{Strong-dif-dominante}, there exists  $D\in \mathcal{D}(G)$ which is  a $\partial_s(G)$-set. 
Since $\gamma(G)\le |D|$, Remark \ref{eq-2} leads to 
$
\partial_s(G)=\n(G)-|D|-|D_w|\le \n(G)-\gamma(G).
$
 Hence, $|D|=\gamma(G)$ and $|D_w|=0$, which implies that $D=D_s$ and so $D$ is also a $2$-dominating set of $G$. Therefore, $\gamma_{_2}(G)\leq |D|=\gamma(G)\leq \gamma_{_2}(G)$, which leads to $\gamma_{_2}(G)=\gamma(G)$.  

Conversely, if $\gamma_{_2}(G)=\gamma(G)$, then  Theorem \ref{Prop-RElat-Differentials} leads to $\partial_s(G)=\n(G)-\gamma(G)$, which completes the proof. 
\end{proof}

Our next result describes the structure of some particular  $\partial_s(G)$-sets for graphs with  $\partial_s(G)=\n(G)-\gamma_{_2}(G)$.

\begin{proposition}\label{CondNecSufEqualityDifs} 
Given a  graph $G$,  the following statements are equivalent.
 \begin{enumerate}[{\rm (i)}]
 \item $\partial_s(G)=\n(G)-\gamma_{_2}(G)$. 
 \item There exists a $\partial_s(G)$-set $D$ which is a dominating set and $D_w=\varnothing$.
 \end{enumerate}
\end{proposition}
 
\begin{proof} If there exists a $\partial_s(G)$-set $D$ which is a dominating set with $D_w=\varnothing$, then $D$ is a $2$-dominating set. Hence, $\partial_s(G)=\partial_s(D)=|N_e(D)|=\n(G)-|D|\leq\n(G)-\gamma_{_2}(G)$. By Theorem~\ref{Prop-RElat-Differentials} we deduce that $\partial_s(G)=\n(G)-\gamma_{_2}(G)$. 

Conversely, assume   $\partial_s(G)=\n(G)-\gamma_{_2}(G)$. For any $\gamma_{_2}(G)$-set $X$ we have that $X_w=\varnothing$, and so $\partial_s(G)=\n(G)-\gamma_{_2}(G)=|N_e(X)|=|N_e(X)|-|X_w|=\partial_s(X)$. This implies that $X$ is a $\partial_s(G)$-set satisfying the required conditions,    which completes the proof.   
\end{proof}

Next we consider the case of graphs with $\partial_s(G)=\n(G)-\gamma_{_2}(G)$.

%%%%%%%
\begin{proposition}\label{CondNecSufEquality3} 
Given a  graph $G$,  the following statements are equivalent.
 \begin{enumerate}[{\rm (i)}]
 \item $\partial_s(G)=\n(G)-2\gamma(G)$. 
 \item Every $\gamma(G)$-set $D$ is a $\partial_s(G)$-set with  $D_s=\varnothing$.
 \end{enumerate}
\end{proposition}
 
\begin{proof}
First, assume $\partial_s(G)=\n(G)-2\gamma(G)$, and let $D$ be a $\gamma(G)$-set. 
Since $|N_e(D)|=\n(G)-\gamma(G)$, we have that
$\n(G)-2\gamma(G)+|D_s|=|N_e(D)|-|D|+|D_s|=\partial_s(D) \le \partial_s(G)=\n(G)-2\gamma(G),$ which implies that $D_s=\varnothing$ and $\partial_s(D) = \partial_s(G)$.

Conversely, if very $\gamma(G)$-set $D$ is a $\partial_s(G)$-set and  $D_s=\varnothing$, then 
$\partial_s(G) = \partial_s(D)=|N_e(D)|-|D|=\n(G)-2\gamma(G)$.
\end{proof}

%%%%%
By Theorem \ref{Prop-RElat-Differentials} and Proposition  
\ref{CondNecSufEquality3}  we deduce the following result.

\begin{proposition} 
If there exists a $\gamma(G)$-set $D$ such that  $D_s\ne \varnothing$, then $$\partial_s(G)\ge \n(G)-2\gamma(G)+1.$$ 
\end{proposition}
%%%

The bound above is tight. For instance, it is achieved by any  graph $G$ of order $r+2$ obtained from a star $K_{1,r}$ by subdividing  one edge. In such a case, $\partial_s(G)=r-1= \n(G)-2\gamma(G)+1.$

\begin{theorem}\label{BoundCover,gamma1}
Let $G$ be a graph. If $\delta(G)\ge 2$, then
$$\partial_s(G)\geq \frac{1}{2}\left(\n(G)-\gamma(G)\right).$$
\end{theorem}

\begin{proof}
Let $S$ be a $\gamma(G)$-set and $C\subseteq V(G)\setminus S$ the set of isolated vertices of the graph  $G[V(G)\setminus S]$. Assume $\delta(G)\ge 2$. If $V(G)=S\cup C$, then $S_w=\varnothing$ and so
$$\partial_s(G)\geq \partial_s(S)=|N_e(S)|=\n(G)-\gamma(G)\geq \frac{1}{2}(\n(G)-\gamma(G)). $$

Now, assume $V(G)\setminus (S\cup C)\neq \varnothing$. Let $S'$ be a $\gamma(G[V(G)\setminus (S\cup C)])$-set. Since $G[V(G)\setminus (S\cup C)]$ does not have isolated vertices, $|S'|\leq \frac{1}{2}(\n(G)-(|S|+|C|))\leq \frac{1}{2}(\n(G)-\gamma(G))$. Let $D=S\cup S'$. Notice that $D$ is a dominating set of $G$ and $D_w=\varnothing$. Hence,
\begin{align*}
\partial_s(G)&\ge \partial_s(D)\\
             &=|N_e(D)|\\
             &=\n(G)-|D|\\
             &\geq \n(G)-\left(\gamma(G)+\frac{1}{2}(\n(G)-\gamma(G))\right)\\
             &=\frac{1}{2}(\n(G)-\gamma(G)).
\end{align*} 
Therefore, the result follows.
\end{proof}

%--------------- Corona --------------

In order to show a class of graphs with $\partial_s(G)=\frac{1}{2}(\n(G)-\gamma(G))$, we consider the case of 
corona graphs. Given two graphs $G_1$ and $G_2$, the {corona product graph} $G_1\odot G_2$ is the graph obtained from $G_1$ and $G_2$, by taking one copy of $G_1$ and $\n(G_1)$ copies of $G_2$ and joining by an edge every vertex from the $i^{th}$-copy of $G_2$ with the $i^{th}$-vertex of $G_1$. For every $x\in V(G_1)$, the copy of $G_2$ in $G_1\odot G_2$ associated to $x$ will be denoted by $G_{2,x}$. The following result was presented without proof in \cite{BermudoStrongProduct}. The reader is referred to  \cite{PerfectDifferential} for a detailed proof. 

\begin{proposition}{\rm \cite{BermudoStrongProduct}}\label{DifferentialCorona}
For any graph $G_1$ and any nontrivial graph $G_2$,
$$\partial(G_1\odot G_2)=\n(G_1)(\n(G_2)-1).$$
\end{proposition}

From the previous result we can easily deduce the following proposition. 

\begin{proposition}\label{StrogDifferentialCorona}
For any graph $G_1$ and any nontrivial graph $G_2$,
$$\partial_s(G_1\odot G_2)=\n(G_1)(\n(G_2)-1).$$
\end{proposition}
\begin{proof}
By Theorem \ref{Prop-RElat-Differentials2} and Proposition  \ref{DifferentialCorona} we deduce that $$\partial_s(G_1\odot G_2)\ge \n(G_1)(\n(G_2)-1).$$
Now, let $S$ be a $\partial_s(G_1\odot G_2)$-set and $S'=\{x\in V(G_1) : \, S\cap (\{x\}\cup V(G_{2,x}))\ne \varnothing\}$. Thus, 
$$ \partial_s(G_1\odot G_2)=|N_e(S)|-|S_w|\le (\n(G_1)-|S'|)+\sum_{x\in S'}(\n(G_2)-1)\le \n(G_1)(\n(G_2)-1).$$
Therefore, the result follows.
\end{proof}

The bound given in Theorem \ref{BoundCover,gamma1} is achieved  by  any corona graph of the form $G\cong G_1\odot K_2$, where $G_1$ is an arbitrary graph. In this case, $\gamma(G)=\n(G_1)$, so that 
$\partial_s(G)=\n(G_1)=\frac{1}{2}(\n(G)-\gamma(G)).$

%To show a family of graphs for which the bound above is achieved, we take  $\mathcal{G}'$ as the family of graphs of order at least three, where $\{\mathcal{L}(G),\mathcal{S}(G)\}$ is a partition of $V(G)$ and every support vertex is  adjacent to exactly two leaves. In this case, for every $G\in \mathcal{G}'$ we have that $\gamma(G)=|\mathcal{S}(G)|=\sigma(G)$ and $\n(G)=3|\mathcal{S}(G)|$. Hence,  by Theorem \ref{BoundCover,gamma1}, $\partial_s(S)\ge \frac{1}{2}(\n(G)-\gamma(G))=|\mathcal{S}(G)|$. Now, since Theorem \ref{Prop-RElat-Differentials} leads to $\partial_s(S)\le \n(G)-\gamma(G)-\sigma(G)=|\mathcal{S}(G)|$,  the equality $\partial_s(S)= \frac{1}{2}(\n(G)-\gamma(G))$ holds.

\subsection{Bounds in terms of domination number and differential}

\begin{theorem}\label{Prop-RElat-Differentials2}
For any  graph $G$, 
$$\partial(G)  \le \partial_s(G)\le \partial(G)+\gamma(G)-1.$$ 
\end{theorem}
\begin{proof}
It is well known that for every graph $G$,
\begin{equation}\label{EqLoweBoundDifferential}
\partial(G)\ge \n(G)-2\gamma(G).
\end{equation}
From Theorem \ref{Prop-RElat-Differentials} and Eq.\ \eqref{EqLoweBoundDifferential} we have that
$\partial_s(G)\le \n(G)-\gamma(G)\le \partial(G)+\gamma(G)$. Now, if $\partial_s(G)= \partial(G)+\gamma(G)$, then $\partial(G)=\n(G)-2\gamma(G)$ and $\partial_s(G)= \n(G)-\gamma(G)$. Hence, by Proposition \ref{CharactEqGamma}, $\gamma_2(G)=\gamma(G)$. Thus, any $\gamma_{_2}(G)$-set $S$ is an independent set, and  so for any $v\in S$, we have that 
$N_e(S\setminus \{v\})=V(G)\setminus S$. This implies that
$$\partial(G)\ge \partial(S\setminus \{v\})=|N_e(S\setminus \{v\})|-|(S\setminus \{v\})|=
\n(G)-2\gamma(G)+1,$$ which is a contradiction. Therefore, 
$\partial_s(G)\le  \partial(G)+\gamma(G)-1$.

On the other side, for any $\partial(G)$-set $D$, 
$$\partial(G)=\partial(D)= |N_e(D)|-|D|\le |N_e(D)|-|D_w|= \partial_s(D)\le \partial_s(G).$$
Therefore, the result follows.
\end{proof}

The upper bound $\partial_s(G)\le \partial(G)+\gamma(G)-1$ is tight. For instance, it is achieved by the complete bipartite graph $K_{2, r}$.

Next we proceed to discuss the structure of some particular  $\partial_s(G)$-sets for graphs with $\partial_s(G)=\partial(G)$.

\begin{proposition}\label{CondNecSufEqualityDifs2} 
Given a  graph $G$,  the following statements are equivalent.
 \begin{enumerate}[{\rm (i)}]
 \item $\partial_s(G)=\partial(G)$. 
 \item There exists a $\partial_s(G)$-set $D$ such that $D_s=\varnothing$.
 \end{enumerate}
\end{proposition}
 
\begin{proof}
 If there exists a $\partial_s(G)$-set $D$ with $D_w=D$, then 
$$\partial_s(G)=\partial_s(D)=|N_e(D)|-|D_w|= |N_e(D)|-|D|=\partial(D)\le \partial(G).$$ Hence, Theorem \ref{Prop-RElat-Differentials2} leads to $\partial_s(G)=\partial(G)$. 

Conversely, assume $\partial_s(G)=\partial(G)$. 
If $D'$ is a  $\partial(G)$-set, then  $$|N_e(D')|-|D'_w|= \partial_s(D')\le  \partial_s(G)=\partial(G)=|N_e(D')|-|D'|\le |N_e(D')|-|D'_w|,$$  and since $D'_w\subseteq D'$, we conclude that $D'=D'_w$ and also $D'$ is a $\partial_s(G)$-set,   which completes the proof.   
\end{proof}

 To conclude this subsection we proceed to derive a result for the particular case of trees. To this end, we need to recall the following known result. 
 
\begin{theorem}\label{BoundTrees} {\rm \cite{CHELLALI201622}}
For any tree $T$,
$$\gamma_{_I}(T)\ge \frac{3}{4}\gamma_{_R}(T).$$
\end{theorem}

 The following bound is a direct consequence of Theorems \ref{Th-Gallai-Roman-Strong}
and \ref{BoundTrees}.

\begin{theorem}  For any tree $T$,
$$\partial_s(T)\leq \left\lfloor \frac{\n(T)+3\partial(T)}{4} \right\rfloor.$$
\end{theorem}

The bound above is tight. For instance, it is achieved by any path of order $6k+3$, as $\partial_s(P_{6k+3})=3k+1$ and $\partial(P_{6k+3})=2k+1$, which implies that $\partial_s(P_{6k+3})=\left\lfloor \frac{6k+3+3\partial(P_{6k+3})}{4} \right\rfloor$.

%Notice that  if $\partial_s(G)=\partial(G)$, then   every  $\partial(G)$-set $D$ satisfies $D_w=D$, but obviously the converse does not hold. For instance,  we can take the cycle graph $C_4$,  where for any $\partial(C_4)$-set $D$ we have that  $|D_w|=|D|=1$, while $\partial_s(C_4)=2>1=\partial(C_4)$.

%%----------------------------------

\subsection{Bounds in terms of order and maximum degree}
\label{SubsectionMaxDegree}

Next we present another basic result obtained in  \cite{CHELLALI201622}.  

\begin{theorem}\label{LowerBound-Chellai}
{\rm \cite{CHELLALI201622}}  For any connected graph $G$, $\gamma_{_I}(G)\ge \frac{2\n(G)}{\Delta(G)+2}$. 
\end{theorem}

It is not difficult to observe that the result above works for non-connected graphs.  

\begin{theorem}\label{Prop-RElat-Differentials3}
For  any  graph $G$,
$$ \Delta(G)-1   \le \partial_s(G)\le \frac{\n(G)\Delta(G)}{\Delta(G)+2}.$$ 
\end{theorem}
\begin{proof}
To prove the lower bound we only need to observe that for every vertex $v$ of maximum degree,  $\partial_s(G)\ge \partial_s(\{v\})=\Delta(G)-1$.

%We proceed  to deduce the  upper bound. 
%By Lemma \ref{Strong-dif-dominante}, there exists  $D\in \mathcal{D}(G)$ which is  a $\partial_s(G)$-set.  Let $X=\cup_{v\in D_w}\epn(v,D)$. Notice that $|X|\leq \Delta(G)|D_w|$. Moreover, since every vertex $v\in N_e(D)\setminus X$ satisfies that $|N(v)\cap D_s|\geq 2$, we obtain that $2|N_e(D)\setminus X|\leq \Delta(G)|D_s|$. Hence, since $D\in \mathcal{D}(G)$,
%\begin{align*}
%\n(G)=&|D|+|N_e(D)|\\
%\leq &|D_w|+|D_s|+ \Delta(G)|D_w|+\frac{1}{2}\Delta(G)|D_s|\\
%\leq &(\Delta(G)+1)|D_w| + \frac{1}{2}(\Delta(G)+2)|D_s|\\
%\leq &\frac{1}{2}(\Delta(G)+2)(2|D_w|+|D_s|)\\
%= &\frac{1}{2}(\Delta(G)+2)(\n(G)-\partial_s(G)).
%\end{align*}
%Therefore, $\partial_s(G)\leq \frac{\n(G)\Delta(G)}{\Delta(G)+2}$, as desired.
The upper bound is a direct consequence of Theorems \ref{Th-Gallai-Roman-Strong} and \ref{LowerBound-Chellai}.
\end{proof}

%---------------------------------
It is well known that $\gamma_{_2}(G)\ge\frac{2\n(G)}{\Delta(G)+2}$ for every graph $G$. The next result shows that the case $\gamma_{_2}(G)=\frac{2\n(G)}{\Delta(G)+2}$ characterises the graphs with $\partial_s(G)= \frac{\n(G)\Delta(G)}{\Delta(G)+2}$. Consequently, these graphs satisfy 
$\partial_s(G)=\n(G)-\gamma_{_2}(G)$.

\begin{proposition} \label{CharactEqBoundMaxDegree}
Given a graph $G$,  the following statements are equivalent.
 \begin{enumerate}[{\rm (i)}]
 \item $\partial_s(G)= \frac{\n(G)\Delta(G)}{\Delta(G)+2}$.
 \item $\gamma_{_2}(G)=\frac{2\n(G)}{\Delta(G)+2}$. 
 \end{enumerate}
\end{proposition}

\begin{proof}
If $\gamma_{_2}(G)=\frac{2\n(G)}{\Delta(G)+2}$, then Theorems \ref{Prop-RElat-Differentials} and \ref{Prop-RElat-Differentials3} immediately lead to $\partial_s(G)= \frac{\n(G)\Delta(G)}{\Delta(G)+2}$.

In order to prove the converse, let  $D$ be a $\partial_s(G)$-set  which satisfies Lemma \ref{Strong-dif-dominante}.  Let $X=\cup_{v\in D_w}\epn(v,D)$. Notice that $|X|\leq \Delta(G)|D_w|$. Moreover, since every vertex $v\in N_e(D)\setminus X$ satisfies that $|N(v)\cap D_s|\geq 2$, we obtain that $2|N_e(D)\setminus X|\leq \Delta(G)|D_s|$. Hence, since $D\in \mathcal{D}(G)$,
\begin{align*}
\n(G)=&|D|+|N_e(D)|\\
\leq &|D_w|+|D_s|+ \Delta(G)|D_w|+\frac{1}{2}\Delta(G)|D_s|\\
\leq &(\Delta(G)+1)|D_w| + \frac{1}{2}(\Delta(G)+2)|D_s|\\
\leq &\frac{1}{2}(\Delta(G)+2)(2|D_w|+|D_s|)\\
= &\frac{1}{2}(\Delta(G)+2)(\n(G)-\partial_s(G)).
\end{align*}
Thus, if $\partial_s(G)= \frac{\n(G)\Delta(G)}{\Delta(G)+2}$, then we have equalities in the previous inequality chain, which implies that $D_w=\varnothing$. Therefore, by Proposition \ref{CondNecSufEqualityDifs}, $\frac{\n(G)\Delta(G)}{\Delta(G)+2}=\partial_s(G)=\n(G)-\gamma_{_2}(G)$, which implies that $\gamma_{_2}(G)=\frac{2\n(G)}{\Delta(G)+2}$.  
\end{proof}

To conclude this subsection we characterize the trees for which the bound $\partial_s(G)\ge \Delta(G)-1$ is achieved. To this end, we need to introduce some additional notation. 

Given a tree $T$ and $v\in V(T)$, the {eccentricity} of $v$ is denoted by  $ecc(v)$. We also   define the following sets associated to $v$,
$$\mathcal{L}(v)=N(v)\cap \mathcal{L}(T) \,\, \text{ and  }\,\, \mathcal{S}_2(v)=\{u\in N(v):\, \deg(u)=2\}.$$

We say that a tree $T$ belongs to the family $\mathcal{T}$ if the following conditions hold for every vertex $v\in V(T)$ with $\deg(v)=\Delta(T)$.
\begin{itemize}
\item[{\rm A.1:}] $v\in \mathcal{S}(T)$ and $ecc(v)\le 3$. 
\item[{\rm A.2:}] $\deg(u)\le 3$ and $|\mathcal{S}_2(u)|\le 1$ for every $u\in N(v)$.
\item[{\rm A.3:}] $\deg(u)\le 2$ for every $u\in V(T)\setminus N[v]$.

\item[{\rm A.4:}]  $|\mathcal{L}(v)|\ge 2$ or there exists $u\in N(v)\cap \mathcal{S}(T)$ such that $\deg(u)=2$.
\end{itemize}

\begin{lemma}\label{LemmaTrees}
Let $T\in \mathcal{T}$ and $v\in V(T)$. If  $\deg(v)=\Delta(T)$, then there exists a $\partial_s(T)$-set $D$ such that $v\in D_w$. 
\end{lemma}
\begin{proof}

Suppose to the contrary that $v\notin D_w$ for every $\partial_s(T)$-set $D$. Let us fix a $\partial_s(T)$-set $D$. Notice that $v\in \mathcal{S}(T)$, by A.1. If $|\mathcal{L}(v)|\geq 2$, then $\partial_s(\{v\}\cup (D\setminus \mathcal{L}(v))\geq \partial_s(D)$, which is a contradiction. Now, if $|\mathcal{L}(v)|=1$, then by A.4   there exists $u\in N(v)\cap \mathcal{S}(T)$ such that $\deg(u)=2$. So, for $X=\{v\}\cup \mathcal{L}(u)\cup D\setminus (\mathcal{L}(v)\cup\{u\})$ we have that $\partial_s(X)\geq \partial_s(D)$, which is again a contradiction, as $X$ is a $\partial_s(T)$-set and $v\in X_w$. Therefore, the result follows. 
\end{proof}

\begin{proposition}
Let $T$ be a  tree with $\n(T)\ge 3$. Then  $\partial_s(T)=\Delta(T)-1$ if and only if $T\in \mathcal{T}$.
\end{proposition}

\begin{proof}
Assume $\partial_s(T)=\Delta(T)-1$ and let $v\in V(T)$ with $\deg(v)=\Delta(T)$. First, we proceed to prove that A.1 holds. If $v\not\in \mathcal{S}(T)$, then $\partial_s(T)\geq \partial_s(V(T)\setminus N(v))= \Delta(T)$, which is a contradiction. Hence, $v\in \mathcal{S}(T)$. Now, if $ecc(v)\geq 4$, then there exists a vertex $v'\in V(T)$, which is at distance three from $v$, such that $\deg(v')\geq 2$. Notice that  $\partial_s(\{v,v'\})\geq \Delta(T)$, which is again a contradiction. Hence, $ecc(v)\leq 3$, as desired. 

Now, we proceed to prove that A.2 holds. Suppose that there exists $u\in N(v)$ such that $\deg(u)\geq 4$. In this case,  $\partial_s(\{v,u\})\geq \Delta(T)$, which is a contradiction. Now,  if there exists $u\in N(v)$ such that $|\mathcal{S}_2(u)|\ge 2$, then $\partial_s(\{v\}\cup N(\mathcal{S}_2(u)) )\ge \Delta(T)$, which is a contradiction. Therefore, A.2 follows.

From A.1 and A.2, if $T$ is a path, then $T\cong P_3$ or $T\cong P_4$. Hence, from now on we assume that $\Delta(T)\ge 3$. 
 
Now, if there exists $u\in V(T)\setminus N[v]$ such that $\deg(u)\geq 3$, then  $\partial_s(\{v,u\})\geq \Delta(T)$, which is a contradiction. Therefore,  A.3 follows.

Finally,  we proceed to prove that A.4 holds. Suppose  to the contrary  that
$|\mathcal{L}(v)|=1$ and for every   $u\in N(v)\setminus \mathcal{L}(v)$ either $\mathcal{S}_2(u)\ne\varnothing$ or $\mathcal{S}_2(u)=\varnothing$ and $\deg(u)=3$. Let $U'=\{u\in N(v):\, \mathcal{S}_2(u)\ne\varnothing\}$ and 
$U''=N(v)\setminus (U'\cup \mathcal{L}(v))$. Notice that if $u\in U'$, then $|\mathcal{S}_2(u)|=1$ and $|\mathcal{L}(u)|\le 1$, by A.2. Also, if $u\in U''$, then $|\mathcal{L}(u)|=2$. Now, let $X=X'\cup U''$, where $$X'=\mathcal{L}(v)\cup \left(\bigcup_{u\in U'}(\{u\}\cup \mathcal{S}_2(u)\cup \mathcal{L}(u))\right).$$
Since $X_s=X'$ and $X_w=U''$, we have that $\partial_s(X)=|N_e(X)|-|X_w|=(1+|X'|+2|U''|)-|U''|=1+|X'|+|U''|=\Delta(v)$, which is a contradiction. Therefore,  
A.4 follows.

%According to the previous analysis, we conclude that $T\in \mathcal{T}$.

Conversely, assume that $T\in \mathcal{T}$. By Theorem \ref{Prop-RElat-Differentials3}, we only need to prove that $\partial_s(T)\le \Delta(T)-1$. As above, let $v\in V(T)$ with $\deg(v)=\Delta(T)$. By Lemma \ref{LemmaTrees}, there exists 
$\partial_s(T)$-set $D$ such that $v\in D_w$. We take $D$ of minimum cardinality among these sets.   Now, if $N(v)\cap D=\varnothing$, then A.1 and A.3 lead to $|N(x)\cap D |\le 1$ for every $x\in  V(T)\setminus N[v]$, which implies that $\partial_s(T)=\partial_s(D)\le \partial_s(\{v\})=\Delta(T)-1$. Now, suppose to the contrary that there exists  $v'\in N(v)\cap D$.  Let $T_{v'}$ be  the sub-tree of $T$ with root $v'$ obtained from $T$ by removing the edge $\{v,v'\}$. By A.2 we have that $\deg(v')\le 3$ and $|\mathcal{S}_2(v')|\le 1$. 
Now, if $\deg(v')\le 2$, then by A.1 either $T_{v'}$ is a trivial tree or $T_{v'}\cong P_2$. Hence, $\partial_s(T)=\partial_s(D)<\partial_s(D\setminus (\{v'\}\cup \mathcal{L}(v')))$, which is a contradiction. 
Finally, if $\deg(v')=3$, 
then either $T_{v'}\cong P_3$ or $T_{v'}\cong P_4$, and so  $\partial_s(T)=\partial_s(D)\le\partial_s(D\setminus V(T_{v'}))$, which is a contradiction, as $|D|>|D\setminus V(T_{v'})|$. 
Therefore, the result follows.
\end{proof}

%-------------

\subsection{Bounds in terms of vertex cover number  and independence number}

If $G$ is a graph with $\delta(G)\ge 2$, then any vertex cover $S$ is a $2$-dominating set, as $V(G)\setminus S$ is an independent set. Hence, $\delta(G)\ge 2$ leads to $\n(G)-\gamma_{_2}(G)\ge \n(G)-\beta(G)=\alpha(G)$. Therefore, the following result is a direct consequence of  Theorem \ref{Prop-RElat-Differentials}.

\begin{corollary}\label{CorollaryLOwerBoundCover}
If $G$ is a graph with $\delta(G)\ge 2$, then $\partial_s(G)\ge \alpha(G)$.
\end{corollary}

The bound above is tight. For instance, the graph $G$  shown in Figure \ref{Fig-Example-1} satisfies $\partial_s(G)=4= \alpha(G)$.

The following result, which states that 
$\partial_s(G)\geq \frac{1}{2}(\n(G)-\alpha(G))=\frac{1}{2}\beta(G),$ is not restricted to the case of graphs with $\delta(G)\ge 2$.

\begin{theorem}\label{BoundCover,gamma}
Let $G$ be a graph. If every component of $G$ has maximum degree at least two, then $$\partial_s(G)\geq \frac{1}{2}\beta(G).$$
\end{theorem}

\begin{proof}
Assume that every component of $G$ has maximum degree at least two. Let $S$ be an $\alpha(G)$-set containing all leaves. Let $C\subseteq V(G)\setminus S$ be the set of isolated vertices of  $G[V(G)\setminus S]$. If $V(G)=S\cup C$, then $$\partial_s(G)\geq \partial_s(S)=|N_e(S)|-|S_w|=\n(G)-\alpha(G)\geq \frac{1}{2}(\n(G)-\alpha(G))=\frac{1}{2}\beta(G). $$

Assume $V(G)\setminus (S\cup C)\neq \varnothing$. Let $S'$ be a $\gamma(G[V(G)\setminus (S\cup C)])$-set. Since the subgraph $G[V(G)\setminus (S\cup C)]$ does not have isolated vertices, $$|S'|\leq \frac{1}{2}(\n(G)-(|S|+|C|))\leq \frac{1}{2}(\n(G)-\alpha(G)).$$ Now, let $D=S\cup S'$. Notice that $D$ is a dominating set of $G$ and $D_w=\varnothing$. Hence,
\begin{align*}
\partial_s(G)&\ge \partial_s(D)\\
             &=|N_e(D)|\\
             &=\n(G)-|D|\\
             &\geq \n(G)-\left(\alpha(G)+\frac{1}{2}(\n(G)-\alpha(G))\right)\\
             &=\frac{1}{2}(\n(G)-\alpha(G))\\
             &=\frac{1}{2} \beta(G).
\end{align*} 
Therefore, the result follows.
\end{proof}

The bound given in Theorem \ref{BoundCover,gamma} is achieved  by  any corona graph of the form $G\cong G_1\odot K_2$, where $G_1$ is an arbitrary graph. In this case, $\beta(G)=2\n(G_1)$ and $\alpha(G)=\n(G_1)$, so that 
$\partial_s(G)=\n(G_1)=\frac{1}{2} \beta(G)=\frac{1}{2}(\n(G)-\alpha(G)).$

%%%%%%%%%%%
\subsection{Bounds in terms of order and semitotal domination number}

A {semitotal dominating set} of a graph $G$ without isolated vertices, is a dominating set $D$ of $G$ such that every vertex in $D$ is within distance two of another vertex of $D$. The {semitotal domination number}, denoted by $\gamma_{_{t2}}(G)$, is the minimum cardinality among all semitotal dominating sets of $G$. This parameter was introduced by Goddard et al.\ in \cite{GoddardSemiTotal}, and it was also studied in \cite{Semitotal2016a,Semitotal2016,Semitotal2019}.

Since every semitotal dominating set is also a dominating set, and every $2$-dominating set is a semitotal dominating set,  the next inequality chain holds.

\begin{equation}\label{eq-relation-domination-parameters}
\gamma(G)\leq \gamma_{_{t2}}(G)\leq \gamma_{_2}(G).
\end{equation}

The following result improves the upper bound given in Theorem \ref{Prop-RElat-Differentials} for graphs without isolated vertices.

\begin{theorem}\label{Semitotal}
For any graph without isolated vertices,
$$\partial_s(G)\leq \n(G)-\gamma_{_{t2}}(G).$$
\end{theorem} 

\begin{proof}
By Lemma \ref{Strong-dif-dominante}, there exists  $D\in \mathcal{D}(G)$ which is  a $\partial_s(G)$-set. We consider a  set $X\subseteq V(G)$ of minimum cardinality among the sets satisfying that $D\subseteq X$ and for any $x\in D_w$, there exists  $u\in N(x)\cap X$. Since $D$ is a dominating set of $G$, it is readily seen that $X$ is a semitotal dominating set and also $|X|\le|D_s|+2|D_w|$. Hence, Remark \ref{eq-2} leads to 
$$
\partial_s(G)=\n(G)-(|D_s|+2|D_w|)\leq \n(G)-|X|\le \n(G)-\gamma_{_{t2}}(G).
$$
\end{proof}

The bound above is achieved by the graph shown in Figure\ref{Fig-Example-2}. In this case, $\gamma_{_{t2}}(G)=6$ and $\partial_s(G)=\n(G)-\gamma_{_{t2}}(G)=8$.

The following result is an immediate consequence of  Theorems \ref{Prop-RElat-Differentials} and  \ref{Semitotal}.
 
 \begin{proposition}\label{Corollary=gammat2-gamma2}
Let  $G$ be a graph.  If $\gamma_{_{t2}}(G)=\gamma_{_2}(G)$, then $\partial_s(G)=\n(G)-\gamma_{_2}(G)$.
 \end{proposition}

%Notice that the case $\gamma_{_{t2}}(G)=2\gamma(G)$ corresponds to graphs without isolated vertices where every $\gamma(G)$-set is a packing.

The converse of Proposition \ref{Corollary=gammat2-gamma2} does not hold. For instance, for the graph $G$ shown in  Figure~\ref{Fig-Example-1} we have that $\partial_s(G)=4=\n(G)-\gamma_{_2}(G)$, while $\gamma_{_2}(G)=4>2=\gamma_{_{t2}}(G)$.

\subsection{Trivial bounds and extreme cases}

In this subsection we discuss the trivial bounds on $\partial_s(G)$ and we characterize the extreme cases.

\begin{proposition}\label{Trivial-Bounds-Perfect-Differential} 
Given  a  graph $G$ of order $\n(G)\ge 3$, the following statements hold.

\begin{enumerate}[{\rm (i)}]
\item $0\le \partial_s(G)\le \n(G)-2.$

\item $\partial_s(G)=0$ if and only if $\Delta(G)\le 1$.

\item $\partial_s(G)=1$ if and only if either $G\cong G_1$ or $G\cong G_1\cup G_2$, where $G_1\in \{C_3,P_3,P_4\}$ and $\Delta(G_2)\leq 1$.

\item $\partial_s(G) = \n(G)-2$ if and only if $\Delta(G)=\n(G)-1$ or $\gamma_{_2}(G)=2$.

\item $\partial_s(G) = \n(G)-3$ if and only if $\gamma_{_2}(G)=3$ and $\Delta(G)\leq \n(G)-2$ or $\gamma_{_2}(G)>3$ and $\Delta(G)=\n(G)-2$.
\end{enumerate}
\end{proposition}
\begin{proof}
Since  $\partial_s(\varnothing)=0$, we have that $\partial_s(G)\ge 0$.  To prove the remaining statements, we take a $\partial_s(G)$-set $D\in \mathcal{D}(G)$, which exists due to Lemma \ref{Strong-dif-dominante}.  Now, since $|D|\ge 1  $, either $|N_e(D)|\le \n(G)-2$ or $|D_w|\ge 1$, which implies that 
$\partial_s(G)=\partial_s(D)\leq \n(G)-2$.

We proceed to prove (ii). If $\Delta(G)\le 1$, then $|N_e(D)|\le |D_w|$. This implies that $\partial_s(G)=\partial_s(D)=0$. 
Conversely,
if $\partial_s(G)=0$, then  $\deg(v)-1=\partial_s(\{v\})\le \partial_s(G)=0$ for every $v\in V(G)$, which implies that $\Delta(G)\le 1$. Therefore, (ii) follows.

We next proceed to prove (iii). If $\partial_s(G)=1$, then (i) and (ii)  lead to $\Delta(G)=2$. If $G$ is connected, then it is easy to see that $G\in \{C_3,P_3,P_4\}$. If $G$ is not connected, then by (ii) we deduce that $G\cong G_1\cup G_2$, where $G_1\in \{C_3,P_3,P_4\}$ and $\Delta(G_2)\leq 1$. The other implication is straightforward. Thus, (iii) follows.

Now, we proceed to prove (iv). If $\Delta(G)=\n(G)-1$, then (i) leads to $\partial_s(G) = \n(G)-2$. Moreover, if  
$\gamma_{_2}(G)=2$, then by Theorem \ref{Prop-RElat-Differentials} and (i) we deduce that $\partial_s(G) = \n(G)-2$.
Conversely, assume 
$\partial_s(G) = \n(G)-2$. In this case, Remark \ref{eq-2} leads to  $|D_s|+2|D_w|=2$, which implies that either $|D_w|=1$ and $|D_s|=0$ or $|D_s|=2$ and $|D_w|=0$. Since $D$ is a dominating set of $G$,  we obtain that $\Delta(G)=\n(G)-1$ or $\gamma_{_2}(G)=2$, as desired.  Therefore, (iv) follows.

Finally, we proceed to prove (v). If either $\gamma_{_2}(G)=3$ and $\Delta(G)\leq \n(G)-2$ or $\gamma_{_2}(G)>3$ and $\Delta(G)=\n(G)-2$, then by Theorem \ref{Prop-RElat-Differentials}, and the statements (i) and (iv) we deduce that $\partial_s(G)=\n(G)-3$. Conversely, assume that $\partial_s(G) =\n(G)-3$.
From (i) and (iv) we deduce that $\Delta(G)\leq \n(G)-2$ and $\gamma_{_2}(G)\geq 3$.  In this case, Remark \ref{eq-2} leads to $|D_s|+2|D_w|=3$, which implies that either $|D_w|=|D_s|=1$ or $|D_s|=3$ and $|D_w|=0$. If  $|D_w|=|D_s|=1$, then $\Delta(G)=\n(G)-2$, while if $|D_s|=3$ and $|D_w|=0$, then $\gamma_{_2}(G)=3$, which completes the proof. 
\end{proof}

%%%%% ---------------------------------

%-----------------------------------------

\section{Concluding remarks}\label{SectionConsequences-Gallai}

In this paper, we have introduced  the theory of strong differentials in  graphs. This new branch of the domination theory, allows us to develop the theory of Italian domination without the use of functions. 
This interesting connection is due to Theorem~\ref{Th-Gallai-Roman-Strong}, which is a Gallai-type theorem. 
For this reason, this article presents the challenge of obtaining new results on Italian domination following this novel approach. As an example, the following table summarizes some of those results obtained here. The first column describes the result that  combined with Theorem \ref{Th-Gallai-Roman-Strong} leads to the result on the second column.

\vspace{0,3cm}
\begin{center}
\begin{tabular}{|l|c|}
\hline 
\rule[-1ex]{0pt}{4ex} \cellcolor{gray!10} From & \cellcolor{gray!10} Result \\ 
\hline 
\rule[-1ex]{0pt}{4ex} Theorem \ref{Prop-RElat-Differentials} & $   \gamma_{_I}(G)\ge \gamma(G)+\sigma(G)$ \\ 
\hline 
\rule[-1ex]{0pt}{4ex} Proposition \ref{CharactEqGamma} & $\gamma_{_I}(G)=\gamma(G)$ $\longleftrightarrow$ $\gamma_{_2}(G)=\gamma(G)$\\ 
%\hline 
%\rule[-1ex]{0pt}{4ex} Proposition \ref{CondNecSufEqualityDifs}  & $\gamma_{_I}(G)=\gamma_{_2}(G)$ $\longleftrightarrow$ there exists a $\partial_s(G)$-set $D\in \mathcal{D}(G)$ with $D_w=\varnothing$ \\ 
\hline 
\rule[-1ex]{0pt}{4ex} Theorem \ref{BoundCover,gamma1} & 
 $\delta(G)\ge 2$ $\longrightarrow$
$\gamma_{_I}(G)\leq \frac{1}{2}\left(\n(G)+\gamma(G)\right)$ \\ 
\hline 
\rule[-1ex]{0pt}{4ex} Theorem \ref{Prop-RElat-Differentials2} & $\gamma_{_I}(G)\ge \gamma_{_R}(G)-\gamma(G)+1$ \\ 
%\hline 
%\rule[-1ex]{0pt}{4ex} Proposition 3.10 & $\gamma_{_I}(G)= \gamma_{_R}(G)$ $\longleftrightarrow$ there exists a $\partial_s(G)$-set $D$ with $D_s=\varnothing$\\ 
\hline 
\rule[-1ex]{0pt}{4ex} Proposition \ref{CharactEqBoundMaxDegree} & $\gamma_{_I}(G)=\frac{2\n(G)}{\Delta(G)+2}$ $\longleftrightarrow$ $\gamma_{_2}(G)=\frac{2\n(G)}{\Delta(G)+2}$  \\ 
\hline 
\rule[-1ex]{0pt}{4ex} Corollary \ref{CorollaryLOwerBoundCover} &  $\delta(G)\ge 2$ $\longrightarrow$
$\gamma_{_I}(G)\leq \alpha(G)$ \\ 
\hline 
\rule[-1ex]{0pt}{4ex} Theorem \ref{BoundCover,gamma} & $G$ connected and $\n(G)\ge 3$ $\longrightarrow$
$\gamma_{_I}(G)\leq \frac{1}{2}\beta(G)$  \\ 
\hline 
\rule[-1ex]{0pt}{4ex} Theorem \ref{Semitotal} & 
$\delta(G)\ge 1$ $\longrightarrow$
$\gamma_{_I}(G)\ge \gamma_{_{t2}}(G)$ \\ 
\hline 
\rule[-1ex]{0pt}{4ex} Proposition \ref{Corollary=gammat2-gamma2} & $\gamma_{_{t2}}(G)=\gamma_{_2}(G)  $  $\longrightarrow$ $\gamma_{_I}(G)=  \gamma_{_{2}}(G)$\\ 
\hline 
\end{tabular} 
\end{center}

% ------------------------------------------------------------------------
\end{document}